\documentclass[11pt]{amsart}
\usepackage{amssymb,amsmath}
\usepackage[T1]{fontenc}

\usepackage{anysize}
\marginsize{2.5cm}{2.5cm}{2.5cm}{2.5cm}
\theoremstyle{plain}
\newtheorem{theorem}{Theorem}
\newtheorem{lemma}{Lemma}

\theoremstyle{remark}
\newtheorem{remark}{Remark}

\title{Perfect powers in products of terms of elliptic divisibility sequences}
\author{Lajos Hajdu}
\address{Institute of Mathematics\\University of Debrecen\\ P.O. Box 400.\\H-4002 Debrecen\\ Hungary}
\email{hajdul@science.unideb.hu}
\author{Shanta Laishram}
\address{Stat-Math Unit, Indian Statistical Institute\\
7, S. J. S. Sansanwal Marg, New Delhi, 110016, India}
\email{shanta@isid.ac.in}
\author{M\'arton Szikszai}
\address{Institute of Mathematics\\University of Debrecen\\ P.O. Box 400.\\H-4002 Debrecen\\ Hungary}
\email{szikszai.marton@science.unideb.hu}


\begin{document}
\thanks{L. Hajdu was supported in part by the OTKA grant K115479.}
\thanks{S. Laishram was supported by in parts by INSA, India, HAS, Hungary and FWF (Austrian Science Fund) grant No. P24574.}
\begin{abstract}
Diophantine problems involving recurrence sequences have a long history and is an actively studied topic within number theory. In this paper, we connect to the field by considering the equation
\[
B_mB_{m+d}\dots B_{m+(k-1)d}=y^\ell
\]
in positive integers $m,d,k,y$ with $\gcd(m,d)=1$ and $k\geq 2$, where $\ell\geq 2$ is a fixed integer and $B=(B_n)_{n=1}^\infty$ is an elliptic divisibility sequence, an important class of non-linear recurrences. We prove that the above equation admits only finitely many solutions. In fact, we present an algorithm to find all possible solutions, provided that the set of $\ell$-th powers in $B$ is given. (Note that this set is known to be finite.) We illustrate our method by an example.
\end{abstract}

\maketitle

\section{Introduction}
Finding perfect powers among the terms or the products of terms of recurrence sequences is a classical Diophantine problem. The case of linear recurrences has a vast literature already. We only mention several important results, without going into details. Peth\H{o} \cite{petho} and independently Shorey and Stewart \cite{shoreystewart} showed that any non-degenerate binary recurrence can admit only finitely many perfect powers and their sizes are effectively bounded. Further, in case when a general linear recurrence of order $k$ has a so-called dominant root, Shorey and Stewart \cite{shoreystewart} proved that the sequence cannot contain a $q$-th power if $q$ is large enough. These results, together with other general theorems concerning the perfect powers among the terms (see e.g. the book of Shorey and Tijdeman \cite{shoreytijdeman} and the references there) suggest that the effective determination of perfect power terms is possible, at least in principle. However, listing all of them for an individual sequence is a highly non-trivial problem. For instance, it was just recently that Bugeaud, Mignotte and Siksek \cite{bugeaudmignottesiksek}, applying modular techniques, came up with a result that gives all perfect powers in the sequences of Fibonacci and Lucas numbers. Note that these are the most basic examples of binary recurrences. For perfect powers in products of terms, the situation is roughly the same. Results for certain infinite families of sequences promise effective determination of all solutions, but usually the bounds are so high that explicit computation cannot be carried out. Concerning the general setting, we mention the paper of Luca and Shorey \cite{lucashorey}, where they gave an effective upper bound for the size of the solutions to the equation when a product of terms from a Lucas sequence or from its companion sequence equals a perfect power. In case of individual recurrences, we refer to Bravo, Das, Guzm\'an and Laishram \cite{bravodasguzmanshanta} who considered the previously mentioned equations with the Pell and Pell-Lucas sequences, listing all solutions. Their proofs also provide a method for Lucas and their companion sequences, in general. For more details on these topics, we point the reader to the above mentioned papers and the references given there.

It is natural to investigate analogous problems for non-linear recurrences. One of the classical and most studied family of such recurrences is given by the elliptic divisibility sequences. The notion of elliptic divisibility sequence was introduced by Ward \cite{ward} as a class of non-linear recurrences satisfying certain arithmetic properties. It is important to note that some special cases of his definition give back Lucas sequences. We follow Silverman \cite{silverman}, whose definition is a conventional and widely used one. Take an elliptic curve $E$ over $\mathbb{Q}$ and a point $P\in E(\mathbb{Q})$ of infinite order. We can write the multiples of $P$ as
\[
nP=\left(\dfrac{A_n}{B_n^2},\dfrac{C_n}{B_n^3}\right)
\]
with integers $A_n,B_n,C_n$ such that $\gcd(A_nC_n,B_n)=1$ and $B_n>0$. (Note that the assumption $B_n>0$ is made only for convenience.) The sequence $B=(B_n)_{n=1}^\infty$ is called an elliptic divisibility sequence. Due to their relation with elliptic curves and various applications, such sequences have attracted increased attention for the last few decades. For example, Shipsey \cite{shipsey} and Swart \cite{swart} established connections between elliptic divisibility sequences and the elliptic curve discrete logarithm problem, while Stange \cite{stange} applied them and their generalizations, the so-called elliptic nets, in the computation of the Weil and Tate pairings. As an exotic application, Poonen \cite{poonen} used them to prove the undecidability of Hilbert's tenth problem over certain rings of integers. In this paper, we are interested in a Diophantine problem concerning perfect powers represented as products of terms of elliptic divisibility sequences.

Questions about finiteness and effective determination of perfect powers among the terms of elliptic divisibility sequences themselves have already been considered by several authors and various results appeared in this direction. Let us take an elliptic divisibility sequence $B=(B_n)_{n=1}^\infty$, an integer $\ell\geq 2$ and introduce the notation
\[
\mathcal{P}_\ell(B)=\{i:B_i\ \textnormal{is an $\ell$-th power}\}.
\]
For later use, also set
$$
N_\ell=|\mathcal{P}_\ell(B)|\ \ \ \text{and}\ \ \ M_\ell=\max\limits_{i\in \mathcal{P}_\ell(B)} i.
$$
Everest, Reynolds and Stevens \cite{everestreynoldsstevens} showed  finiteness for the set $\mathcal{P}_\ell(B)$, however, their proof is ineffective and hence does not give an upper bound for the size of its elements. Further, they noted that under the assumption of the $abc$-conjecture one can let the exponent $\ell$ vary and prove finiteness for the set of all perfect powers in the sequence. As in the case of linear recurrences, listing the elements of $\mathcal{P}_\ell(B)$ is a highly non-trivial problem. A  paper of Reynolds \cite{reynolds} explains a procedure to find every perfect power in the sequence when $B_1$ is divisible by $2$ or $3$. There are more explicit results for square and cube terms by Bizim and Gezer \cite{bizimgezer1, bizimgezer2}. (Note that their definition of elliptic divisibility sequence differs from ours, since it involves a torsion point rather than a point of infinite order.)

Let $B=(B_n)_{n=1}^\infty$ be an elliptic divisibility sequence such that $B_1=1$ and $\ell\geq 2$ is fixed. We will point out later in the Introduction that $B_1=1$ is unnecessary, but makes the presentation smoother. Consider the diophantine equation
\begin{equation}
\label{eq1}
B_{m}B_{m+d}\dots B_{m+(k-1)d}=y^\ell
\end{equation}
in positive integers $m,d,k,y$ with $k\geq 2$ and $\gcd(m,d)=1$. We prove that \eqref{eq1} admits only finitely many solutions. Further, we bound $m,d,k,y$ in terms of $N_\ell$ and $M_\ell$. In fact, our method provides an algorithm to find all the solutions to equation \eqref{eq1}, whenever ${\mathcal P}_\ell(B)$ is given explicitly.

\begin{theorem}
\label{theorem1}
Let $\ell\geq 2$ be a fixed integer. Then, equation \eqref{eq1} has only finitely many solutions. Further, there exists an effectively computable constant $c_1(N_\ell,M_\ell)$ depending only on $N_\ell$ and $M_\ell$ such that $\max(m,d,k,y)<c_1(N_\ell,M_\ell)$. In particular, if $\mathcal{P}_\ell(B)$ is given then all solutions to \eqref{eq1} can be effectively determined.
\end{theorem}

To prove Theorem \ref{theorem1}, we need to combine several tools, including arithmetic properties of elliptic divisibility sequences, arguments from \cite{bravodasguzmanshanta, lucashorey} and new variants of bounds, developed in this paper, concerning the greatest prime divisor and the number of prime divisors of blocks of consecutive terms of arithmetic progressions.

Finally, we mention a possible generalization of \eqref{eq1}, which could be handled by our arguments. In their paper, Everest, Reynolds and Stevens \cite{everestreynoldsstevens} remark that it is possible to modify their proof on the finiteness of $\mathcal{P}_\ell(B)$ to deduce finiteness also for $S$-unit multiples of $\ell$-th powers, where $S$ is any given finite set of primes. Then with slight changes (but more technicality involved) we could prove the analogue of Theorem \ref{theorem1} for the equation
\[
B_mB_{m+d}\dots B_{m+(k-1)d}=by^\ell,
\]
where $b$ is an arbitrary $S$-unit, i.e. $b$ is composed of fixed primes (coming from $S$) with unspecified non-negative exponents. Observe that it also makes the assumption $B_1=1$ unnecessary. Indeed, dividing both sides by $B_1^k$, we get an equation of the form
\[
B'_mB'_{m+d}\dots B'_{m+(k-1)d}=b'y^\ell.
\]
Since the sequence $B'=(B'_n)_{n=1}^\infty=(B_n/B_1)_{n=1}^\infty$ preserves the arithmetic properties of $B$ we rely on (see Remark \ref{remarkcondition}), one can solve the above more general equation, as well (and hence omit the condition $B_1=1$).

\section{Auxiliary tools}

Recall that throughout the paper we use the assumption $B_1=1$. Thus, in particular, we have ${\mathcal P}_\ell(B)\neq \emptyset$, $N_\ell\geq 1$ and $M_\ell\geq 1$.

Arithmetic properties of elliptic divisibility sequences have been well-studied, see for instance the fundamental paper of Ward \cite{ward} and theses of Shipsey \cite{shipsey} and Swart \cite{swart} and the references given there. Let $B=(B_n)_{n=1}^\infty$ be an elliptic divisibility sequence, $p$ be a prime and denote by $r_p$ the smallest number such that $p\mid B_{r_p}$ holds. Then $r_p$ is called the rank of apparition of $p$ in $B$. Further, let $\nu_p(z)$ stand for the exponent of $p$ in $z$.

\begin{lemma}
\label{lemma1}
Let $B=(B_n)_{n=1}^\infty$ be an elliptic divisibility sequence. Then we have the following properties.
\begin{itemize}
\item[(i)] If $p\mid B_m$, then
\[
\nu_p(B_m)=\nu_p\left(\dfrac{m}{r_p}\right)+\nu_p(B_{r_p}).
\]
\item[(ii)] $B$ is a strong divisibility sequence, that is, for every $m,n\geq 1$ we have
\[
\gcd(B_{m},B_{n})=B_{\gcd(m,n)}.
\]
\item[(iii)] For every prime $p$ we have
\[
r_p\leq p+1+2\sqrt{p}.
\]
\item[(iv)] For $m\mid n$ we have
\[
\gcd\left(B_m,\dfrac{B_n}{B_m}\middle)\right| \dfrac{n}{m}.
\]
\end{itemize}
\end{lemma}

\begin{proof}
For (i) see formula (13) in \cite{silverman}. Part (ii) is exactly Theorem 6.4 in \cite{ward} and also follows from (i), while (ii) is an immediate consequence of the famous Hasse-Weil theorem, see Section 4.7.2 in \cite{shipsey}. Applying (iii) for $m\mid n$ yields
\[
\nu_p\left(\dfrac{B_n}{B_m}\right)= \nu_p\left(\dfrac{n}{r_p}\right)+\nu_p(B_{r_p})- \nu_p\left(\dfrac{m}{r_p}\right)-\nu_p(B_{r_p})= \nu_p\left(\dfrac{n}{m}\right)
\]
and hence
\[
\min\left(\nu_p(B_m), \nu_p\left(\dfrac{B_n}{B_m}\right)\right)\leq \nu_p\left(\dfrac{n}{m}\right)
\]
which proves part (iv).
\end{proof}

We write $P(z)$ for the greatest prime divisor of the positive integer $z$, with the convention $P(1)=1$. Further, for $0\leq i<k$ we put
\[
m+id=a_ix_i
\]
with $P(a_i)\leq k$ and $\gcd\left(x_i,\prod\limits_{p\leq k}p\right)=1$.

Our next lemma plays a crucial role later on.
As we are not aware of such a result appearing in the literature, we give its simple proof, as well.

\begin{lemma}
\label{lemma2}
Let $0\leq i<k$. Then
\[
\gcd\left(B_{x_i},\prod\limits_{j\neq i}B_{m+jd}\right)=1\quad \mathrm{and}\quad \gcd\left(B_{x_i},\dfrac{B_{m+id}}{B_{x_i}}\middle)\right| a_i.
\]
\end{lemma}

\begin{proof}
If $x_i=1$, then the assertion of the lemma follows from $B_1=1$. Thus assume that $x_i\neq 1$. Then for every $p\mid x_i$ we have $p>k$. Since a prime greater than $k$ can divide at most one of $m,m+d,\dots,m+(k-1)d$, for every $j\neq i$ we get $\gcd(x_i,m+jd)=1$ and from part (i) of Lemma \ref{lemma1} the first formula follows. The second part of the statement is an immediate consequence of part (iv) of Lemma \ref{lemma1}.
\end{proof}

Using the above lemmas, we can already prove Theorem \ref{theorem1} for small values of $k$.

\begin{lemma}
\label{lemma3}
Let $(m,d,k,y)$ be a solution to \eqref{eq1} with $k\leq 48$. Then we have $\max(m,d)\leq c_2M_\ell$, where $c_2=1$ for $k\leq 16$, $c_2=2$ for $17\leq k\leq 24$ and $c_2=3$ for $25\leq k\leq 48$.
\end{lemma}

\begin{proof} Suppose first that $k\leq 16$. Then by a classical result of Pillai \cite{pillai} there is a term $m+id$ with $\gcd(m+id,m+jd)=1$ for every $j\neq i$. Observe that here we may assume that $i>0$. Indeed, if $i=0$ then by $\gcd(m,m+jd)=1$ for all $j=1,\dots,k-1$, using Pillai's result again for the terms $m+d,\dots,m+(k-1)d$, we can find an index $i>0$ with the desired property. Then, by $B_1=1$ and part (i) of Lemma \ref{lemma1} we have $\gcd(B_{m+id},B_{m+jd})=1$. Hence $m+id\in\mathcal{P}_\ell(B)$ and $\max(m,d)\leq m+d\leq m+id\leq M_\ell$, and the lemma follows in this case.

Assume next that $17\leq k\leq 24$. Then by Theorem 2.2 of Hajdu and Saradha \cite{hajdusaradha} there is a term $m+id$ with $\gcd(m+id,m+jd)\leq 2$ for every $j\neq i$. Similarly as in the case $k\leq 16$, we may assume that $i>0$. If in fact $\gcd(m+id,m+jd)=1$ for all $j\neq i$, then just as before, we get $m+id\in\mathcal{P}_\ell(B)$ and $\max(m,d)\leq M_\ell$. So we may assume that $\gcd(m+id,m+jd)=2$ for some $j\neq i$; in particular, $a_i$ is even. Write $a_i=2t$, and observe that $\gcd(t,m+jd)=1$ for all $j\neq i$. Rewrite \eqref{eq1} as
\begin{equation}
\label{eq2}
B_{tx_i}\dfrac{B_{m+id}}{B_{tx_i}}\prod\limits_{j\neq i}B_{m+jd}=y^\ell .
\end{equation}
Observe that $\gcd(tx_i,m+jd)=1$ and hence $\gcd(B_{tx_i},B_{m+jd})=1$ for every $j\neq i$. On the other hand, by part (iv) of Lemma \ref{lemma1} we have
\[
\gcd\left(B_{tx_i},\dfrac{B_{m+id}}{B_{tx_i}}\middle)\right| 2.
\]
Now if $2\mid B_{tx_i}$, then we have $r_2\mid tx_i$. This by $r_2\leq 5$ following from part (ii) of Lemma \ref{lemma1}, implies that $r_2\mid t$. However, this would clearly contradict the choice of $m+id$. So $B_{tx_i}$ is odd, and hence coprime to $B_{m+id}/B_{tx_i}$. Thus \eqref{eq2} yields that $tx_i\in\mathcal{P}_\ell(B)$ and we get $\max(m,d)\leq m+id=2tx_i\leq 2M_\ell$, proving our claim also in this case.

Finally, assume that $25\leq k\leq 48$. Then, using again Theorem 2.2 of \cite{hajdusaradha}, by a similar argument as before we obtain that there is an $i>0$ such that $\gcd(m+id,m+jd)\leq 3$ for every $j\neq i$. Now if this $\gcd$ is in fact $\leq 2$ for all $j\neq i$, then the same argument as for $17\leq k\leq 25$ gives $\max(m,d)\leq 2M_\ell$. Hence we may assume that there is a $j\neq i$ such that $\gcd(m+id,m+jd)=3$. In particular, $3\mid a_i$, and we can write $a_i=3t$. Now we can just follow the argument for $17\leq k\leq 24$ to conclude that $tx_i\in\mathcal{P}_\ell(B)$ and get $\max(m,d)\leq 3M_\ell$. This finishes the proof.
\end{proof}

\begin{remark}
In certain cases, Lemma \ref{lemma3} can be extended for larger values of $k$. This is based on quantitites concerning a problem of Pillai \cite{pillai} and its generalizations, obtained by Hajdu and Saradha \cite{hajdusaradha} and by Hajdu and Szikszai \cite{hajduszikszai1, hajduszikszai2}. To do so, one needs to know which terms $B_n$ satisfy $B_n=1$ and compare the set of the corresponding indices with the tables in \cite{hajduszikszai1, hajduszikszai2}. For example, if we take the sequence generated by the point $P=(0,0)$ on the curve $y^2+y=x^3-x$, then we have $B_1=B_2=B_3=B_4=B_6=1$. Using Table 2 in \cite{hajduszikszai2} we could extend Lemma \ref{lemma3} for $k\leq 78$.
\end{remark}

Fix now $m$, $d$ and $k$ and consider the indices $m+id$ $(0\leq i<k)$. Write $k'=k+1+2\sqrt{k}$ and put
\[
\begin{array}{llllll}
W_1 & = & \{i\ :\ \exists p\mid (m+id)\ \mathrm{with}\ p>k\}, & w_1 & := & |W_1|;\\
W_2 & = & \{i\in W_1\ :\ \exists p\mid (m+id)\ \mathrm{with}\ k<p\leq k'\}, & w_2 & := & |W_2|;\\
W_0 & = & W_1\setminus W_2, & w_0 & := & |W_0|.
\end{array}
\]
Here $p$ always denotes a prime number. Clearly, we have $w_0=w_1-w_2$. Further,
$$
w_2\leq \pi_d(k')-\pi_d(k)\leq \pi(k')-\pi(k),
$$
where $\pi_d(x)$ stands for the number of primes up to $x$ which does not divide $d$.

An important connection between the sets $W_0$ and $\mathcal{P}_\ell(B)$ is given by the following lemma.

\begin{lemma}
\label{lemma4}
Let $(m,d,k,y)$ be a solution to \eqref{eq1}. Then $x_i\in\mathcal{P}_\ell(B)$ for each $i\in W_0$. In particular, we have $w_0\leq N_\ell$, and also $k<M_\ell$ if $w_0>0$.
\end{lemma}

\begin{proof}
Observe that for $i\in W_0$ the numbers $x_i$ are distinct, and also that we have $q>k'$ for every prime divisor $q$ of $x_i$. Let $i\in W_0$ and let $p$ be a prime divisor of $a_i$. Then
by part (ii) of Lemma \ref{lemma1} we have $r_p\leq p+1+2\sqrt{p}\leq k'$. Thus $r_p\nmid x_i$, whence $p\nmid B_{x_i}$, and by Lemma \ref{lemma2} we have $\gcd(B_{x_i},B_{m+id}/B_{x_i})=1$. This immediately gives $x_i\in\mathcal{P}_\ell(B)$. As the $x_i$ are distinct for $i\in W_0$, we obtain $w_0\leq N_\ell$. Finally, as if $i\in W_0$ then we have $k<x_i\leq M_\ell$, the lemma follows.
\end{proof}

\begin{remark}
\label{remarkcondition}
Concerning properties of elliptic divisibility sequences, Lemma \ref{lemma4} is the last we state. With little effort one can prove that the sequence $B'=(B'_n)_{n=0}^\infty=(B_n/B_1)_{n=0}^\infty$ preserves (i) even if $B_1\neq 1$. Hence (ii) and (iv) also remain valid. Since (iii) is true for arbitrary curves (Hasse's theorem holds), we find that the statements of Lemma \ref{lemma1} are independent of the condition $B_1=1$. This also implies the truth of Lemma \ref{lemma2} and \ref{lemma4} for $B'$. As it was mentioned already in the Introduction, this allows one to omit $B_1=1$ and consider \eqref{eq1} without restrictions on $B$.
\end{remark}

In what follows, we shall establish lower bounds for $w_0$. For this, we need results concerning the number of terms $W(\Delta)$ of $\Delta$ having a prime factor $>k$, where
$$\Delta=m(m+d)\dots (m+(k-1))d.$$

\begin{lemma}
\label{newlemma}
Let $k\geq 31$. Then we have
\begin{itemize}
\item[(i)] $W(\Delta)\geq\min\left(\left\lfloor \frac{3}{4}\pi(k)\right\rfloor-1,\pi(2k)-\pi(k)-1\right)$ if $d=1$ and $m>k$,
\item[(ii)] $W(\Delta)>\pi(2k)-\pi_d(k)-\rho$ if $d>1$, where
$\rho=1$ for $d=2$ and $\rho=0$ otherwise.
\end{itemize}
\end{lemma}

\begin{proof} Part $(i)$ immediately follows from Corollary 1 of \cite{shantashorey2}. Though the assertion was stated for the number of
distinct prime factors of $\Delta$, it is in fact valid for $W(\Delta)$ as given by the proof.  Part $(ii)$ is a simple consequence of Theorem 1 of \cite{shantashorey1}.
\end{proof}

We also use estimates for $\pi(x)$, due to Rosser and Schoenfeld \cite{rs}.

\begin{lemma}
\label{lemmapix}
For any $x\geq 17$ we have
$$
\frac{x}{\log x}<\pi(x)< \frac{x}{\log x}\left(1+\frac{3}{2\log x}\right).
$$
\end{lemma}

\begin{proof} The upper bound is part of Theorem 1 of \cite{rs}, while the lower bound is in Corollary 1 in the same paper.
\end{proof}

Lemma \ref{newlemma} combined with Lemma \ref{lemmapix} easily implies the following assertion.

\begin{lemma}
\label{lemma5}
Let $k\geq 2$. Further, assume that $m>k$ if $d=1$. Then there exists an absolute constant $c>0$ such that
$$
w_0>\frac{ck}{\log k}.
$$
\end{lemma}

\begin{proof} Recall that $w_0=w_1-w_2$ and $w_2\leq \pi_d(k+1+2\sqrt{k})-\pi_d(k)\leq \pi(k+1+2\sqrt{k})-\pi(k)$. By observing that $w_1\geq W(\Delta)$, the
assertion follows from Lemmas \ref{newlemma} and \ref{lemmapix} by a
simple calculation.
\end{proof}

Under a certain assumption, we can establish a much better lower bound for $w_0$.

\begin{lemma}
\label{lemma6}
Let $k\geq 48$, and assume that $m+d\geq (k-1)^4$. Then we have
$$
w_0\geq \frac{3(k-1)}{4}-\pi_d(k+1+2\sqrt{k}).
$$
\end{lemma}

\begin{proof}
We follow standard arguments, going back to Erd\H{o}s. For similar results, see e.g. \cite{shantashorey1} and the references given there.

For each prime $p\leq k$ and $p\nmid d$, choose an index $i_p$ with $0\leq i_p<k$ such that
$$
\nu_p(m+i_p d)\geq \nu_p(m+id)\ \ \ (i=0,1,\dots,k-1).
$$
Put
$$
I=\{i_p\ :\ p\leq k,\ p\nmid d\},
$$
and write $J$ for the complement of $I\cup W_0\cup\{0\}$ in $\{0,1,\dots,k-1\}$. We clearly have $|J|\geq k-w_1-\pi_d(k)-1$. Let
$$
\Delta'=\prod\limits_{i\in J} (m+id),
$$
and observe that all prime divisors of $\Delta'$ is at most $k$, and also that $(\Delta',d)=1$. Let $p$ be any prime with $p\leq k$ and $p\nmid d$. Then for any $i=0,1,\dots,k-1$ we have
$$
\nu_p(m+id)\leq \nu_p(m+id-(m+i_pd))\leq \nu_p(i-i_p).
$$
This easily gives $\nu_p(\Delta')\leq \nu_p((k-1)!)$, implying
$\Delta'\mid (k-1)!$. Hence we get
$$
(m+d)^{k-w_1-\pi_d(k)-1}\leq (k-1)!.
$$
Now our assumption $m+d\geq (k-1)^4$ yields
$$
w_1\geq \frac{3(k-1)}{4}-\pi_d(k).
$$
Using $w_0=w_1-w_2$ and $w_2\leq \pi_d(k+1+2\sqrt{k})-\pi_d(k)$, the assertion follows.
\end{proof}

\section{Proof of Theorem \ref{theorem1}}

\begin{proof}[Proof of Theorem \ref{theorem1}]
If $k\leq 48$, then the statement is given by Lemma \ref{lemma3}. So we may assume that $k\geq 49$. We split the proof into two parts.

Suppose first that $d>1$, or $d=1$ and $m>k$. Then by Lemmas \ref{lemma4} and \ref{lemma5} we get that $k$ is bounded in terms of $N_\ell$ (and also in terms of $M_\ell$). Now if $m+d\leq (k-1)^4$, then we are done. Otherwise, Lemma \ref{lemma6} gives that
$$
w_0\geq \frac{3(k-1)}4-\pi_d(k+1+2\sqrt{k}).
$$
Now apart from at most $\pi_d(k)$ indices $i$, we have that $\nu_p(a_i)\leq \nu_p((k-1)!)$. (The exceptions are those indices $i_p$ for which $\nu_p(a_{i_p})$ is maximal.) This shows that if
\begin{equation}
\label{neweq}
\frac{3(k-1)}4-\pi_d(k+1+2\sqrt{k})-\pi_d(k)>1,
\end{equation}
then there are at least two indices $i\neq j$ such that all $a_i,a_j,x_i,x_j$ are bounded in terms of $N_\ell$ and $M_\ell$. As one of these indices, say $i$, is positive, by $m+d\leq m+id=a_ix_i$ we obtain that $m$ and $d$ are also bounded in terms of $N_\ell$ and $M_\ell$. A simple calculation based upon Lemma \ref{lemmapix} shows that \eqref{neweq} holds whenever $k\geq 62$. Then, working with the concrete values of the $\pi(x)$ function, we get that \eqref{neweq} holds in fact for $k\geq 42$. Hence the theorem follows in this case.

Assume next that $d=1$ and $m\leq k$. Then there exists an effectively computable constant $c_3=c_3(N_\ell)>0$ depending only on $N_\ell$ such that if $m+k-1>c_3(N_\ell)$, then the interval $\left(\frac{2}{3}(m+k-1),m+k-1\right)$ contains more than $N_\ell$ primes. Observe that by $m\leq k$ these primes are among $m,m+1,\dots,m+k-1$, and further that each of these primes divides exactly one of these numbers. Let $q$ be any of these primes, and write $q=m+i$. Observe that then by part (i) of Lemma \ref{lemma1}, $\gcd(B_{m+i},B_{m+j})=B_1=1$ for any $j\neq i$ with $0\leq j<k$. Hence $m+i\in {\mathcal P}_\ell(B)$. However, since we have more than $N_\ell$ primes among $m,\dots,m+k-1$, this yields a contradiction. Thus $m+k-1\leq c_3(N_\ell)$, and our claim follows also in this case.
\end{proof}

\section{An example}

Consider the elliptic curve $E\ :\ y^2+xy=x^3+x^2-7x+5$ and the elliptic divisibility sequence $B_n=(B_n)_{n=1}^\infty$ generated by the point $P=(2,-3)$. Reynolds \cite{reynolds} found the following perfect powers in $B_n$:
$$
B_1=B_2=B_3=B_4=B_7=1,\ B_{12}=2^7.
$$
Now we illustrate how our method works, assuming that there are no other perfect powers in $B_n$. (Note that once the set of all perfect powers is given, our method describes all solutions to \eqref{eq1}.)

Under the above assumption, we have
\[
\mathcal{P}_\ell(B)=\begin{cases}
\{1,2,3,4,7,12\}, & \mathrm{if}\ \ell=7;\\
\{1,2,3,4,7\}, & \mathrm{otherwise,}
\end{cases}
\]
and hence
\[
N_\ell=\begin{cases}
6, & \mathrm{if}\ \ell=7;\\
5, & \mathrm{otherwise;}\end{cases}
\qquad \mathrm{and}\qquad M_\ell =\begin{cases}
12, & \mathrm{if}\ \ell=7;\\
7, & \mathrm{otherwise.}\end{cases}
\]

Following the proof of Lemma \ref{lemma5}, by a simple calculation we get that for $k\geq 49$ we have $w_0\geq 1$. However, then by Lemma \ref{lemma4} we obtain that $k<M_\ell\leq 12$, a contradiction.

Hence we conclude that $k\leq 48$. Then following the proof of Lemma \ref{lemma3}, we get $m+d\leq 3M_\ell\leq 36$. As $m,d$ and $k$ are small, we can easily check all possibilities. (Note that for this we can work with the {\sl indices} and not with the {\sl terms} of $B_n$ themselves.)  We find that (under our assumption) the only solutions $(m,d,k,y)$ of equation \eqref{eq1} for arbitrary $\ell$ are given by
\[
(1,1,2,1),\ (1,1,3,1),\ (1,1,4,1),\ (1,2,2,1),\ (1,3,2,1),\ (1,3,3,1),\ (1,6,2,1),\ (2,1,2,1),
\]
\[
(2,1,3,1),\ (2,5,2,1),\ (3,1,2,1),\ (3,4,2,1),\ (4,3,2,1)\}
\]
and further, for $\ell=7$, we also have the solutions
\[
(1,11,2,2),\ (2,5,3,2),\ (7,5,2,2).
\]

\section*{Acknowledgements}
The authors are grateful to the referee for his/her useful comments on the paper.

\end{document}